\documentclass{gtart}
\usepackage{amsmath, amssymb, graphicx, epsfig, verbatim, graphs, psfrag, 
pinlabel}
\usepackage{amscd}
\usepackage{amssymb}

\newtheorem{thm}{Theorem}[section]

\newtheorem{cor}[thm]{Corollary}
\newtheorem{lem}[thm]{Lemma}
\newtheorem{prop}[thm]{Proposition}

\newtheorem{rem}[thm]{Remark}

\numberwithin{equation}{section}

\newcommand{\bfz}{{\mathbb {Z}}}

\newcommand{\bfq}{{\mathbb {Q}}}

\newcommand{\mubar}{{\overline {\mu}}}
\newcommand{\Z}{\mathbb Z}

\newcommand{\R}{\mathbb R}
\newcommand{\bfr}{\mathbb R}

\newcommand{\cpkk}{{\overline {{\mathbb C}{\mathbb P}^2}}}
\newcommand{\cpk}{{\mathbb {CP}}^2}

\DeclareMathOperator{\pf}{pf}

\DeclareMathOperator{\genus}{genus}

\DeclareMathOperator{\rank}{rk}

\begin{document}

\title{On symplectic caps}

\author{David T. Gay} 
\address{Euclid Lab\\
428 Kimball Rd,\\
Iowa City, IA 52245, USA}

\author{Andr\'{a}s I. Stipsicz}
\address{R{\'e}nyi Institute of Mathematics\\
Re{\'a}ltanoda utca 13--15, Budapest, Hungary\\
e--mail: d.gay@euclidlab.org and stipsicz@math-inst.hu} 


\begin{abstract}
An important class of contact 3--manifolds are those that arise as
links of rational surface singularities with reduced fundamental
cycle. We explicitly describe symplectic caps (concave fillings) of
such contact 3--manifolds. As an application, we present a new
obstruction for such singularities to admit rational homology disk
smoothings.
\end{abstract}
\maketitle

\begin{center}
{\em We dedicate this paper to Oleg Viro on the occasion of his 60th
  birthday.}
\end{center}

\section{Introduction}
Our understanding of topological properties of (weak) symplectic fillings of
certain contact 3--manifolds showed a dramatic improvement in the recent past.
These developments rested on recent results in symplectic topology, most
notably on McDuff's characterization of (closed) rational symplectic
4--manifolds \cite{McDuff}. In order to apply results of McDuff, however,
\emph{symplectic caps} were needed to close up the fillings at hand. The
general results of Eliashberg and Etnyre \cite{Eli, Et} showed that such caps
do exist in general, but these results can be used powerfully only in case a
detailed description of the cap is also available. This was the case, for
example, for lens spaces with their standard contact structures \cite{Lisca},
or for certain 3--manifolds which can be given as links of isolated surface
singularities \cite{BO, BS, OO}.

In the following we will show an explicit construction of symplectic
caps for contact 3--manifolds which can be given as links (with their
Milnor fillable structures) of rational singularities with reduced
fundamental cycle. In topological terms it means that the 3--manifold
can be given as a plumbing of spheres along a negative definite tree,
with the additional assumption that the absolute value of the framing
at each vertex is at least the valency of the vertex.  The
construction of the cap in this case relies on a symplectic handle
attachment along a component of the binding of a compatible open book
decomposition. In the terminology of open book decompositions, our
construction coincides with the cap-off procedure initiated and
further studied by Baldwin \cite{Bald}.

The success of the rational blow--down procedure (initiated by
Fintushel and Stern~\cite{FSrat} and then extended by
J. Park~\cite{Prat}) led to the search for isolated surface
singularities which admit rational homology disk smoothings. Strong
restrictions on the combinatorics of the resolution graph of such a
singularity were found in \cite{SSW}, and by identifying Neumann's
$\mubar$--invariant with a Heegaard Floer theoretic invariant of the
underlying 3--manifold, further obstructions for the existence of such
a smoothing were given in \cite{S}. More recently the question has
been answered for all singularities with starshaped resolution graphs
(in particular, for weighted homogeneous singularities) in \cite{BS},
but the general problem remained open. Motivated by our construction
of a symplectic cap for special types of Milnor fillable contact
3--manifolds, we show examples of surface singularities which pass all
tests provided by \cite{S, SSW} but still do not admit rational
homology disk smoothings.

The paper is organized as follows. In Section~\ref{sec:two} we
describe the symplectic handle attachment which caps off a boundary
component of a compatible open book
decomposition. Section~\ref{sec:three} is devoted to the detailed
description of the topology of the symplectic cap, and also an example
is worked out. In Section~\ref{sec:four} we show that certain
singularities do not admit rational homology disk smoothings.

{\bf Acknowledgements:} The authors would like to acknowledge support
from the Hungarian-South African Bilateral Project NRF 62124
(ZA-15/2006), and thank Andr\'as N\'emethi for many useful
conversations. The second author was also supported by OTKA T67928.
The authors would also like to thank Chris Wendl for pointing out an important
mistake in the first version of this paper.

\section{Symplectic handle attachments}
\label{sec:two}
Throughout this section suppose that $(Y,\xi )$ is a strongly convex boundary
component of a symplectic $4$--manifold $(X, \omega )$, that $\xi$ is
supported by an open book decomposition with oriented page $\Sigma$, oriented
binding $B = \partial \Sigma$ and monodromy $h$, and that $L$ is a sublink of
$B$. For each component $K$ of $L$, let $\pf(K)$ denote the page--framing of
$K$, the framing induced by the page $\Sigma$. Note that if $Y_L$ is the
result of performing surgery on $Y$ along each component $K$ of $L$ with
framing $\pf(K)$, and if $L \neq B$, then the open book on $Y$ induces a
natural open book on $Y_L$ with page $\Sigma_L$ equal to $\Sigma \cup_L
(\amalg^{|L|} D^2)$, the result of capping off each $K$ with a disk, and with
monodromy equal to $h$ extended by the identity on the $D^2$ caps. (In
\cite{Bald} this construction has been examined from the Heegaard Floer
theoretic point of view.)

If instead, for $|L|=1$ and $L=K$, $Y_K$ is obtained by surgery along $K$ with
framing $\pf(K)\pm 1$, then the open book on $Y$ induces a natural open book
on $Y_K$ with page $\Sigma_K = \Sigma$ and with monodromy $h_K = h \circ
\tau_K^{\mp 1}$, where $\tau_K$ is a right-handed Dehn twist along a circle in
the interior of $\Sigma$ parallel to $K$. In fact, if $K\neq B$ then surgery
with framing $\pf (K)-1$ coincides with Legendrian surgery along a Legendrian
realization of $K$ on the page, hence the 4--dimensional cobordism resulting
from the construction supports a symplectic structure.  In the following two
theorems we extend the existence of such a symplectic structure to the cases
where the surgery coefficients are $\pf (K)$ and $\pf (K)+1$.

In the first case, where the surgery coefficient is $\pf(K)$, we have a rather
technical extra condition in terms of the existence of a closed $1$--form with
certain behavior near $K$. Later we will state one case in which this
condition is always satisfied, but for the moment we leave it technical
because the theorem is most general that way. When we discuss the behavior of
anything near a component $K$ of $B$, we always use oriented coordinates
$(r,\mu,\lambda)$ near $K$ such that $\mu, \lambda \in S^1$ are the meridional
and longitudinal coordinates, respectively, chosen to represent the page
framing. In other words, $\mu^{-1}(\theta)$, for any $\theta \in S^1$, is the
intersection of a page with this coordinate neighborhood, and the closure of
$\lambda^{-1}(\theta)$ is a meridional disk. Also, we assume that
$\partial_{\lambda}$ points in the direction of the orientation of $K$,
oriented as the boundary of the page.

\begin{thm}\label{t:0framed}
  Suppose that $L$ is a sublink of $B$, not equal to $B$, and that $X_L
  \supset X$ is the result of attaching a $2$--handle to $X$ along each
  component $K$ of $L$ with framing $\pf(K)$. Suppose furthermore that there
  exists a closed $1$--form $\alpha_0$ defined on $Y \setminus L$ which, near
  each component $K$ of $L$, has the form $m_K d\mu + l_K d\lambda$ for some
  constants $m_K$ and $l_K$, with $l_K > 0$. (The coordinates
  $(r,\mu,\lambda)$ near $K$ are as described in the preceding paragraph.)
  Then $\omega$ extends to a symplectic form $\omega_L$ on $X_L$ and the new
  boundary $Y_L$ is $\omega_L$--convex. The new contact structure $\xi_L$ is
  supported by the natural open book on $Y_L$ described above.
\end{thm}
\begin{proof}
Let $\pi\colon Y \setminus B \rightarrow S^1$ be the fibration
associated to our given open book on $Y$, and let $\pi_L \colon Y_L
\setminus (B \setminus L) \rightarrow S^1$ be the fibration for the
induced open book on $Y_L$.  Let $Z$ be $[-1,0] \times Y$ together
with the $2$--handles attached along $\{0\} \times L \subset \{0\}
\times Y$, and identify $Y$ with $\{0\} \times Y$. Thus $Z$ is a
cobordism from $\{-1\} \times Y$ to $Y_L$ and $Y \cap Y_L$ is nonempty
and is in fact the complement of a neigborhood of $L$ in $Y$. We will
show that there is a symplectic structure $\eta$ on $Z$ which, on
$[-1,0] \times Y$, is equal to the symplectization of a certain
contact form $\alpha$ on $Y$ supported by $(B,\pi)$ and such that
$Y_L$ is $\eta$--convex, with induced contact structure $\xi_L$
supported by the natural open book $(B \setminus L, \pi_L)$ on $Y_L$
described above. This proves the theorem.

As mentioned above, for each component $K$ of $L$ we use coordinates
$(r,\mu,\lambda)$ on a neighborhood $\nu \cong D^2 \times S^1$ of $K$, with
$(r,\mu)$ being polar coordinates on the $D^2$--factor and $\lambda$ being the
$S^1$--coordinate, in such a way that $\mu = \pi|_\nu$. Thus the pages are the
level sets for $\mu$. We will also add now the convention that $r$ is always
parametrized so as to take values in $[0,1+\epsilon]$ for some small positive
$\epsilon$.

Let $\nu'$ be the corresponding neighborhood in $Y_L$ of the belt--sphere for
the $2$--handle $H_K$ which is attached along $K$, with corresponding
coordinates $(r',\mu',\lambda')$, with the natural diffeomorphism from $\nu
\setminus \{r=0\} \rightarrow \nu' \setminus \{r' = 0\}$ given by $r' = r$,
$\mu' = -\lambda$ and $\lambda' = \mu$. Note that $\pi_L|_{\nu'} = \lambda'$,
which is defined on all of $\nu'$.

There are, of course, many different contact structures supported by the given
open book on $Y$, but they are all isotopic, and, up to isotopy, we can always
assume that $\xi$ has the following behavior in each neighborhood $\nu$ of
each component $K$ of $L$:
\begin{enumerate}
\item $\xi$ is $(\mu,\lambda)$--invariant. I.e. there exist functions $F(r)$
  and $G(r)$ such that $\xi$ is spanned by $\partial_r$ and $F(r) \partial_\mu
  + G(r) \partial_\lambda$. We necessarily have $G(0) = 0$ and we will adopt
  the convention that $F(0) > 0$, so that $G'(0)<0$ and thus $G(r) < 0$ for
  $r$ close to $0$.
\item As $r$ ranges from $0$ to $1$, $\xi$ makes a full quarter turn in the
  $(\mu,\lambda)$ plane. In other words, the vector $(F(r),G(r)) \in \R^2$
  goes from $F(0)>0, G(0)=0$ to $F(1)=0, G(1) < 0$, with $F(r)>0$ and $G(r)<0$
  for all $r \in (0,1)$. (We can make this assumption precisely because $L
  \neq B$. One way to see this is to think of the construction of a contact
  structure supported by a given open book as beginning with a Weinstein
  structure on the page. This Weinstein structure comes from a handle
  decomposition of the page, and if we choose a handle decomposition starting
  with collar neighborhoods of the components of $L$ and then adding
  $1$--handles, we will get the desired behavior.)
\end{enumerate}
So now we assume $\xi$ has the form above.

Next we claim that we can find a contact form $\alpha$ for this $\xi$
satisfying certain special properties. To understand the local properties of
$\alpha$ near each $K$, consider Figure~\ref{f:fggraphs}.
\begin{figure}[ht!]
\labellist
\small\hair 2pt
\pinlabel $\sqrt{2 l_K}$ [r] at 58 132
\pinlabel $R_1$ [r] at 58 105
\pinlabel $1$ [t] at 129 62
\pinlabel $f(r)$ [rb] at 156 156
\pinlabel $R_2$ [r] at 370 147
\pinlabel $1$ [t] at 468 61
\pinlabel $g(r)$ [b] at 507 147
\pinlabel $r$ [l] at 200 62
\pinlabel $r$ [l] at 512 62
\endlabellist
\centering \includegraphics[scale=0.65]{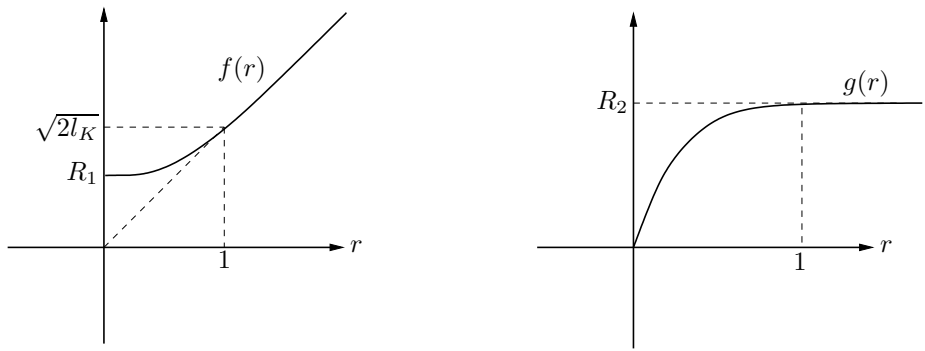}
\caption{Graphs of the functions $f$ and $g$}
\label{f:fggraphs}
\end{figure}
This figure shows graphs of two functions $f$ and $g$, specified by constants
$R_1$, $l_K$ and $R_2$. The properties of $f$ and $g$ are:
\begin{enumerate}
\item The function $f$ is monotone increasing with $f'(0) = 0$
  and $f'(r) > 0$ for $r > 0$.
\item $f(0) = R_1$ and $f(r) = \sqrt{2 l_K} r$ for $r \geq 1$. (Hence $\sqrt{2
    l_K} > R_1$.)
\item $g(0) = 0$. 
\item The function $g$ is monotone increasing with $g'(r) > 0$
  on $[0,1)$.
\item $g(r) = R_2$ for $r \geq 1$.
\end{enumerate}
The claim, then, is that there exists a contact form $\alpha$ for $\xi$ such
that:
\begin{enumerate}
\item The $1$--form $\alpha - \alpha_0$ is a positive contact form on the
  complement of the neighborhoods of radius $r \leq 1$ of each component $K$
  of $L$, and also satisfies the support condition for the given open book
  outside these neighborhoods.
\item For each component $K$ of $L$ there are constants $R_1$ and $R_2$ and
  associated functions $f$ and $g$, as in Figure~\ref{f:fggraphs} (with the
  constant $l_K$ coming from $\alpha_0 = m_K d\mu + l_K d\lambda$), with
  $\frac{1}{2}R_2^2 > m_K$, such that, in the neighborhood $\nu$ of $K$,
  $\alpha$ has the form:
  \[ \alpha = \frac{1}{2}g(r)^2 d\mu + (l_K - \frac{1}{2} f(r)^2) d\lambda \]
  (We might need to reparametrize the coordinate $r$, but only via a
  repara\-metrization fixing $0$ and $1$.)
\end{enumerate}
The condition $\frac{1}{2} R_2^2 > m_K$ is necessary to guarantee that $\alpha
- \alpha_0$ is positive contact when $r \geq 1$, and will also be used later.

To verify this claim, first choose any contact form $\alpha'$ for $\xi$
satisfying the support condition for the given open book. Now note that, for
any suitably large constant $k > 0$, $k \alpha' - \alpha_0$ is a positive
contact form satisfying the support condition. We know that, in $\nu$, $k
\alpha' = -G(r) d\mu + F(r) d\lambda$ for functions $F(r),G(r)$ such that the
vector $(F(r),G(r))$ makes one quarter turn through the fourth quadrant, as
$r$ goes from $0$ to $1$. Because $k$ is large we may assume that $G(1) <
-m_K$. We can then scale $k \alpha'$ by a positive function $\phi(r)$
supported inside $r \leq 1+\epsilon$ so as to arrange that the pair of
functions $(\tilde{F}(r)=\phi(r) F(r), \tilde{G}(r)=\phi(r) G(r))$ has the
appropriate shape and then we let $\frac{1}{2} g(r)^2 = -\tilde{G}(r)$ and
$l_K - \frac{1}{2}f(r)^2 = \tilde{F}(r)$. Then we have $\alpha = \phi(r) k
\alpha'$.

Now embed $\nu$ and $\nu'$ in $\R^4$ as follows, using polar coordinates
$(r_1,\theta_1,r_2,\theta_2)$ on $\R^4$: The embedding of $\nu$ is given by
$(r_1 = f(r), \theta_1 = -\lambda, r_2 = g(r), \theta_2 = \mu)$. The embedding
of $\nu'$ is given by $(r_1 = \sqrt{2 l_K} r', \theta_1 = \mu', r_2 = R_2,
\theta_2 = \lambda')$. This is illustrated in Figure~\ref{f:embeddingsH},
which also shows that the region between $\nu$ and $\nu'$ is precisely our
$2$--handle $H$ attached along $K$ with framing $\pf(K)$. The overlap $\nu
\cap \nu'$ is the set $\{r_1 \geq \sqrt{2 l_K}, r_2 = R_2\}$, which in
$\nu$--coordinates is $\{r \geq 1\}$ and in $\nu'$--coordinates is $\{r' \geq
1\}$.
\begin{figure}[ht!]
\labellist
\small\hair 2pt
\pinlabel $H$ at 191 130
\pinlabel $K$ [bl] at 245 101
\pinlabel $K$ [br] at 94 101
\pinlabel $\nu'$ [b] at 121 158
\pinlabel $\nu'$ [t] at 205 44
\pinlabel $R_2$ [br] at 171 158
\pinlabel $R_1$ [t] at 242 44
\pinlabel $\sqrt{2 l_K}$ [t] at 304 44
\pinlabel $\nu$ [r] at 87 68
\pinlabel $\nu$ [l] at 254 134
\pinlabel $(r_2,\theta_2)$ [B] at 171 231 
\pinlabel $(r_1,\theta_1)$ [l] at 343 101
\endlabellist
\centering
\centering
\includegraphics[scale=0.85]{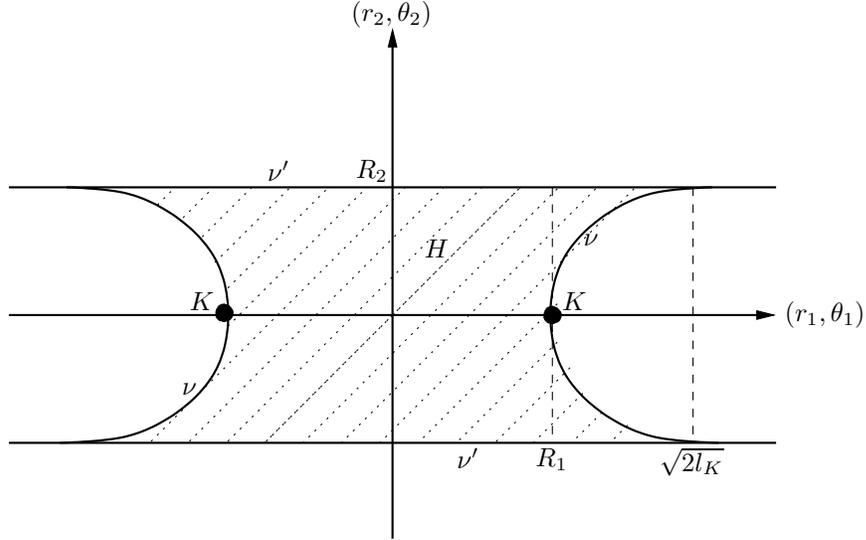}
\caption{Embeddings of $\nu$, $\nu'$ and $H$ into $\R^4$}
\label{f:embeddingsH}
\end{figure}

Consider the standard symplectic form $\omega_0 = r_1 dr_1 d\theta_1 + r_2
dr_2 d\theta_2$ on $\R^4$. Note that $\omega_0 |_\nu = g g' dr d\mu - f f' dr
d\lambda = d\alpha$, so that $H$ equipped with this symplectic form can be
glued symplectically to $[-1,0] \times Y$ with the symplectization of
$\alpha$. Next note that $\omega_0 |_{\nu'} = 2 l_K r' dr' d\mu' = d\alpha'$,
where $\alpha' = \frac{1}{2} (\sqrt{2 l_K} r')^2 d\mu' + (\frac{1}{2} R_2^2 -
m_K) d\lambda'$. (Here we see that $\frac{1}{2} R_2^2 > m_K$ is necessary for
$\alpha'$ to be a positive contact form and to be supported by the open book
inside this neighborhood $\nu'$.) On the overlap $\nu \cap \nu' \subset \R^4$,
using the coordinates $(r,\mu,\lambda)$ from $\nu$, we see that $\alpha' =
(\frac{1}{2} R_2^2 - m_K) d\mu + (-\frac{1}{2} (\sqrt{2 l_K} r)^2) d\lambda =
\alpha - \alpha_0$. Thus we see that $\alpha'$ extends to the rest of $Y_L$ as
$\alpha - \alpha_0$, concluding the proof of the theorem. 
\end{proof}

In fact, 2--handles can be attached with framing $\pf (K)+1$ to boundary
components of a compatible open book, and the symplectic structure will still
extend. In this case, however, the convex boundary will become concave. More
precisely:

\begin{thm}\label{t:1framed}
  Suppose that $K=B$ and that $X_K \supset X$ is the result of attaching a
  $2$-handle $H$ to $X$ along $K$ with framing $\pf(K)+1$. Then $\omega $
  extends to a symplectic form $\omega_K$ on $X_K$ and the new boundary $Y_K$
  is $\omega_K$--concave. The new (negative) contact structure $\xi_K$ is
  supported by the natural open book on $Y_K$ described above.
\end{thm}
\begin{proof}
  This is Theorem~1.2 in~\cite{G}. However in that paper, which predates
  Giroux's work on open book decompositions, the terminology is slightly
  different. Definition~2.4 of~\cite{G} defines what it means for a transverse
  link $L$ in a contact $3$-manifold $(M,\xi)$ to be ``nicely fibered''. It is
  easy to see that if $L$ is the binding of an open book supporting $\xi$ then
  $L$ is nicely fibered. (The notion of ``nicely fibered'' is more general
  because, in open book language, it allows for ``pages'' whose boundaries
  multiply cover the binding.) Theorem~1.2 in~\cite{G} then says that if we
  attach $2$-handles to all the components of a nicely fibered link in the
  strongly convex boundary of a symplectic $4$-manifold, with framings which
  are more positive than the framings coming from the fibration, then the
  symplectic form extends across the $2$-handles to make the new boundary
  strongly concave. In our case we have a single component and we are
  attaching with framing exactly one more than the framing coming from the
  fibration. Finally, Addendum~5.1 of~\cite{G} characterizes the negative
  contact structure induced on the new boundary as follows: There exists a
  constant $k$ such that $\alpha_K = k d\pi - \alpha$ on the complement of the
  surgery knots. (Here we are identifying $Y \setminus K$ with the complement
  in $Y_K$ of the belt sphere for $H$ in the obvious way.) The constant $k$ is
  simply the appropriate constant so that $\alpha_K$ extends to all of $Y_K$.
  Then $d \pi \wedge d\alpha_K = -d\pi \wedge d\alpha$ which is positive on
  $-Y_K$. Since $d\alpha_K = -d\alpha$, and the Reeb vector field for $\alpha$
  is tangent to the level sets for the radial function $r$ on a neighborhood
  of $K$ (see Definition~2.4 in~\cite{G}), the Reeb vector field for
  $\alpha_K$ is necessarily tangent to the new binding of $Y_K$ and it is not
  hard to check that it points in the correct direction, so that $\alpha_K$ is
  supported by the natural open book on $Y_K$. 
\end{proof}

We have the following application. (For a similar result, see
\cite[Theorem 4']{wendl}.)

\begin{cor}\label{c:closedo}
If the open book on $Y$ is planar (i.e. $\genus(\Sigma) = 0$) then
$(X,\omega)$ embeds in a closed symplectic 4--manifold $(Z,\eta)$
which contains a symplectic $(+1)$--sphere disjoint from $X$.
\end{cor}

As preparation we need the following:
\begin{lem} \label{l:1form} Let $B$ be the (disconnected) binding of a planar
  open book on $Y$, and let $L \subset B$ be the complement of a single
  component of $B$. Then there exists a $1$--form $\alpha_0$ on $Y \setminus
  L$ such that, near each component $K$ of $L$, $\alpha_0$ has the form
  $\alpha_0 = m_K d\mu + l_K d\lambda$, for $l_K > 0$. (The coordinates near
  $K$ are as in Theorem~\ref{t:0framed}, and are determined by the open book.)
\end{lem}

\begin{proof}
  Let $Y_L$ be the result of page-framed surgery on $L$, with the
  corresponding oriented link $L' \subset Y_L$ (the cores of the surgeries).
  Note that $Y_L \cong S^3$ because the induced open book on $Y_L$ has disk
  pages. Thus $L'$ is an oriented link in $S^3$ and there exists a map $\sigma
  \colon S^3 \setminus L' \to S^1$ with the closure of each
  $\sigma^{-1}(\theta)$, for each regular value $\theta$, an oriented Seifert
  surface for $L'$. Pull $\sigma$ back to $Y \setminus L = Y_L \setminus L'$
  and let $\alpha_0 = d\sigma$. 
\end{proof}

\begin{proof}[Proof of Corollary~\ref{c:closedo}]
  Let the components of $B$ be $K_1, \ldots, K_n$. Attach $2$--handles to
  $K_1, \ldots, K_{n-1}$ with framings $\pf(K_i)$, as in
  Theorem~\ref{t:0framed}. This gives $(X',\omega') \supset (X,\omega)$ with
  $\omega'$--convex boundary $(Y',\xi')$. Now attach a $2$--handle to $K_n$
  with framing $\pf(K_n)+1$ as in Theorem~\ref{t:1framed}; the resulting
  concave end is $S^3$ with its negative contact structure supported by the
  standard disk open book, i.e. the contact structure is the standard negative
  tight contact structure. Thus we can fill in the concave end with the
  standard symplectic structure on $B^4$. Alternatively, we can note that, on
  $Y'$, the positive contact structure $\xi'$ is supported by an open book
  with page diffeomorphic to a disk. In other words, $Y'$ is diffeomorphic to
  $S^3$ and $\xi'$ is the standard positive tight contact structure on $S^3$.
  Thus we can remove a standard $(B^4,\omega_0)$ from $\cpk$ with its standard
  K\"{a}hler form, and replace $(B^4,\omega_0)$ with $(X',\omega')$ to get
  $(Z,\eta)$. Since there is a symplectic $(+1)$--sphere in $\cpk$ disjoint
  from $B^4$, we end up with a symplectic $(+1)$--sphere in $(Z,\eta)$
  disjoint from $X'$, and hence disjoint from $X$. 
\end{proof}

By~\cite{McDuff} the symplectic 4--manifold $Z$ found in the proof of
Corollary~\ref{c:closedo} is diffeomorphic to a blowup of $\cpk$.  Let $Z'$ be
the result of anti-blowing down the symplectic $(+1)$--sphere in $Z$ (i.e.
$Z'$ is the union of the $4$-manifold $X'$ in the proof of the corollary above
with $B^4$). Then $Z'$ (still containing $X$) is diffeomorphic to the
connected sum of a number of copies of $\cpkk$. Let $D$ be the closure of $Z'
\setminus X$ in $Z'$; we will call this the \emph{dual configuration} (or
\emph{compactification}) for $X$. Thus we get embeddings of the intersection
forms $H_2(X;\Z )$ and $H_2 (D; \Z)$ into a negative definite diagonal
lattice, and therefore both $H_2(X; \Z)$ and $H_2(D;\Z)$ are negative
definite.

\begin{rem} {\rm A very similar compactification has been found by N\'emethi
    and Popescu-Pampu in \cite{NP}, using rather different methods.}
\end{rem}

\section{Examples: rational surface singularities with reduced fundamental 
cycle}
\label{sec:three}

Suppose that $\Gamma$ is a plumbing tree of spheres which is negative
definite, and at each vertex the absolute value of the framing is at least the
number of edges emanating from the vertex. Every negative definite plumbing
graph $\Gamma$ gives rise to a (not necessarily unique) surface singularity,
and the further assumptions on $\Gamma$ ensure that the singularity has
reduced fundamental cycle. According to Laufer's algorithm, for example, this
property implies that the singularity is rational, cf. \cite[Section~3]{S}.
The Milnor fillable contact structure on such a 3--manifold is known to be
compatible with a planar open book decomposition \cite{EtO, EO, Sch}.  A fairly
explicit description of such an open book decomposition can be given by a
construction resting on results of \cite{GaS}. By \cite[Proposition~5.3]{GaS}
the Milnor fillable contact structure is compatible with an open book
decomposition resting on a toric construction (cf. \cite[Section~4]{GaS}), and
therefore by \cite[Proposition~4.2]{GaS} a compatible planar open book can be
explicitly given as follows.

View the tree $\Gamma$ as a planar graph in $\bfr ^2$ and consider the
boundary sphere of an $\epsilon$ neighborhood of it in $\bfr ^3$. Suppose that
$v$ is a vertex of $\Gamma$ with framing $e_v$ and valency $d_v$. Then near
$v$ drill $-e_v-d_v\geq 0$ holes on the sphere. The resulting planar surface
will be the page of the open book decomposition. Consider a parallel circle to
each boundary component, and further curves near each edge, as shown by the
example of Figure~\ref{f:pelda}. The monodromy of the open book decomposition
is simply the product of the right handed Dehn twists defined by all these
curves on the planar surface.
\begin{figure}[ht!]
\centering
\includegraphics[scale=0.65]{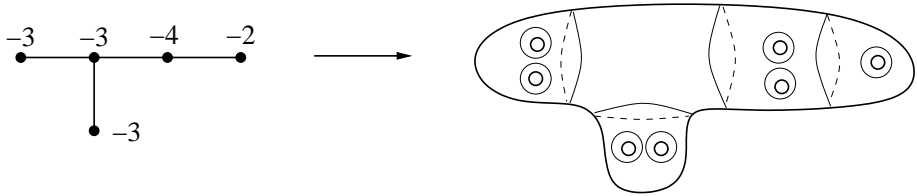}
\caption{Light circles on the punctured sphere define the monodromy of
  the open book}
\label{f:pelda}
\end{figure}

Consider now the Kirby diagram for $Y$ based on the open book decomposition as
follows: regard the planar page as a multipunctured disk. (This step involves
a choice of an 'outer circle'.)  Every hole on the disk defines a 0--framed
unknot linking the boundary of the hole, while the light circles defining the
monodromy through right handed Dehn twists give rise to a $(-1)$--framed
unknots. In fact, the 0--framed unknots can be turned into dotted circles, and
then viewed as 4--dimensional 1--handles (for these notions of Kirby calculus,
see \cite{GS}). These will build up a Lefschetz fibration with fiber
diffeomorphic to the page of the open book, and the addition of the
$(-1)$--framed circles correspond to the vanishing cycles of the Lefschetz
fibration, giving the right monodromy.

Having this Kirby diagram for $Y$, a relative handlebody diagram for the dual
configuration $D$ (built on $-Y$) can be easily deduced by performing
0--surgery along all the boundary circles except the outer one.  This
operation corresponds to capping off all but the last boundary component of
the open book defining the Milnor fillable structure on $Y$. Since after all
the capping off we get an open book with a disk as a page, the 4--manifold $D$
is a cobordism from $-Y$ to $S^3$.

It is usually more convenient to have an absolute handlebody than a relative
one, and since the other boundary component of $D$ is $S^3$, by turning $D$
upside down we can easily derive a handlebody description first for $-D$ and
then, after the reversal of the orientation, for $D$. After appropriate
handleslides, in fact, the diagram for $D$ can be given by a simple algorithm.
Since we only dualize 2--handles, $D$ can be given by attaching 2--handles to
$D^4$.  The framed link can be given by a braid, which is derived from the
plumbing tree by the following inductive procedure. To start, we choose a
vertex $v$ where the strict inequality $-e_v-d_v>0$ holds. (Such a vertex
always exists, for example, we can take a leaf.)  We will choose the outer
circle to be the boundary of one of the holes near $v$.  Now associate to
every inner boundary component a string and to every light circle a box
symbolizing a full negative twist of the strings passing through the box,
which in our case comprise of those strings which correspond to the boundary
components encircled by the light circle.  The framing on a string is given by
the negative of the 'distance' of the boundary component from the outer
circle: this distance is simply the number of light circles we have to cross
when traversing from the boundary component to the outer circle. Another
(obviously equivalent) way of describing the same braid purely in terms of the
graph $\Gamma$ goes as follows: choose again a vertex $v$ with $-e_v-d_v>0$,
and consider $-e_u-d_u$ strings for each vertex $u$, except for $v$ for which
we take only $-e_v-d_v-1$ strings.  Introduce a full negative twist on the
resulting trivial braid (corresponding to the light circle parallel to the
outer circle), and then introduce a further full negative twist for every edge
$e$ in the graph, where the strings affected by the negative twist can be
characterized by the property that they correspond to vertices which are in a
component of $\Gamma - \{ e\}$ not containing the distinguished vertex $v$.
Finally, equip every string corresponding to a vertex $u$ by $r_{uv}-2$ where
$r_{uv}$ is the negative of the minimal number of edges we traverse when
passing from $u$ to $v$.

We will demonstrate this procedure through an explicit family of examples. 
(For a similar result see \cite[Theorem~3]{wahl}.) To
this end, suppose that the graph $\Gamma _n$ is given by Figure~\ref{f:int}.
\begin{figure}[ht]
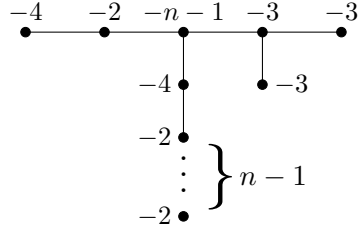

\begin{center}
\setlength{\unitlength}{1mm}
\unitlength=0.7cm
\begin{graph}(10,6)(0,-5)
\graphnodesize{0.2}

 \roundnode{m2}(2,0)
 \roundnode{m3}(3.5,0)
 \roundnode{m4}(5,0)
 \roundnode{m5}(6.5,0)
 \roundnode{m6}(8,0)
 \roundnode{m7}(6.5,-1)  
 \roundnode{m8}(5,-1)
 \roundnode{m9}(5,-2)
 \roundnode{m10}(5,-3.5)

\edge{m3}{m2}
\edge{m3}{m4}
\edge{m4}{m5}
\edge{m6}{m5}
\edge{m7}{m5}
\edge{m8}{m4}
\edge{m9}{m8}

\freetext(5,-2.4){\Large $.$}
\freetext(5,-2.7){\Large $.$}
\freetext(5,-3){\Large $.$}

  \autonodetext{m2}[n]{{\small $-4$}}
  \autonodetext{m3}[n]{{\small $-2$}}
  \autonodetext{m4}[n]{{\small $-n-1$}}
  \autonodetext{m5}[n]{{\small $-3$}}
  \autonodetext{m6}[n]{{\small $-3$}}
  \autonodetext{m7}[e]{{\small $-3$}}
  \autonodetext{m8}[w]{{\small $-4$}}
  \autonodetext{m9}[w]{{\small $-2$}}
  \autonodetext{m10}[w]{{\small $-2$}}

\freetext(5.7,-2.75)
{{\Huge $\rbrace$}}

\freetext(6.7,-2.75){$n-1$}

\end{graph}
\end{center}
\caption{\quad An interesting family of plumbing graphs.}
\label{f:int}
\end{figure}
It is easy to see that the graphs in the family for $n\geq 1$ are all negative
definite, and for $n\geq 2$ define a rational singularity with reduced
fundamental cycle.  Assume that $n\geq 3$ and choose a boundary circle near
the $(-n-1)$--framed vertex to be the outer circle. The page of the planar
open book, together with the light circles (giving rise to the monodromy
through right handed Dehn twists) are pictured by Figure~\ref{f:lap1}
\begin{figure}[ht!]
\centering
\includegraphics[scale=0.65]{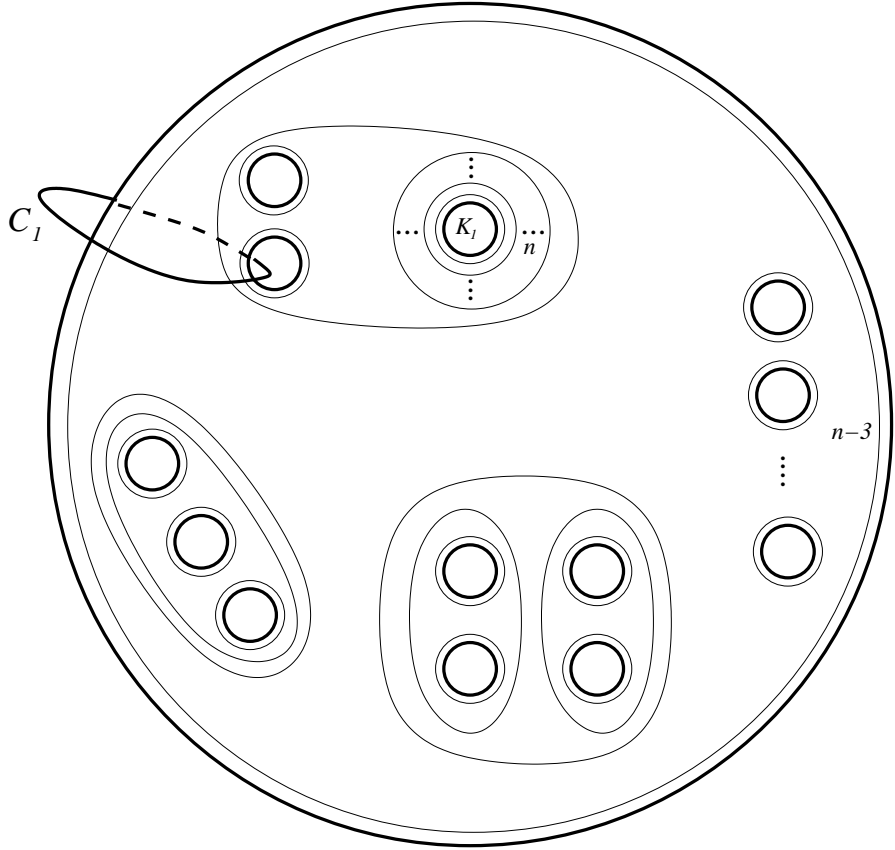}
\caption{The light circles on the disk define the monodromy of the open book.
  There are $n$ concentric light circles around the boundary component
  labelled by $K$ and there are $n-3$ boundary circles on the right hand side
  of the disk. For each of the interior boundary components there should be a
  corresponding unknot $C_i$ linking it and the exterior boundary component;
  here we have only drawn $C_1$.}
\label{f:lap1}
\end{figure}
(with the circle $C_1$ disregarded for a moment).  The 0--framed unknots
originating from the 1--handles of the Lefschetz fibration become unknots
which each link one of the interior boundary components of the punctured disk
once and the exterior boundary once. In the diagram, the unknot labelled $C_1$
is one of these unknots; we have not drawn the rest because they would only
complicate the picture needlessly, but it is important to remember that there
is one such unknot for each interior boundary.  Putting $(-1)$--framings to
all light circles we get a convenient description of $Y$. Now add framing $0$
to all boundary components except the outer one.  The result is a cobordism
$D$ from $-Y$ to $S^3$. Mark all these circles (for example, use the
convention of \cite{GS} by replacing all framing $a$ with $\langle a \rangle$)
and turn $D$ upside down: add 0--framed meridians to the circles corresponding
to the boundary components of the open book (these are the curves along which
we 'capped off' the open book). Now sliding and blowing down marked curves
only, we end up with the diagram of $-D$, and by reversing all crossings and
multiplying all framings by $(-1)$ eventually we get a Kirby diagram for $D$
as it is shown by Figure~\ref{f:d}. (Every box in the diagram means a full
negative twist.)
\begin{figure}[ht!]
\centering
\includegraphics[scale=0.75]{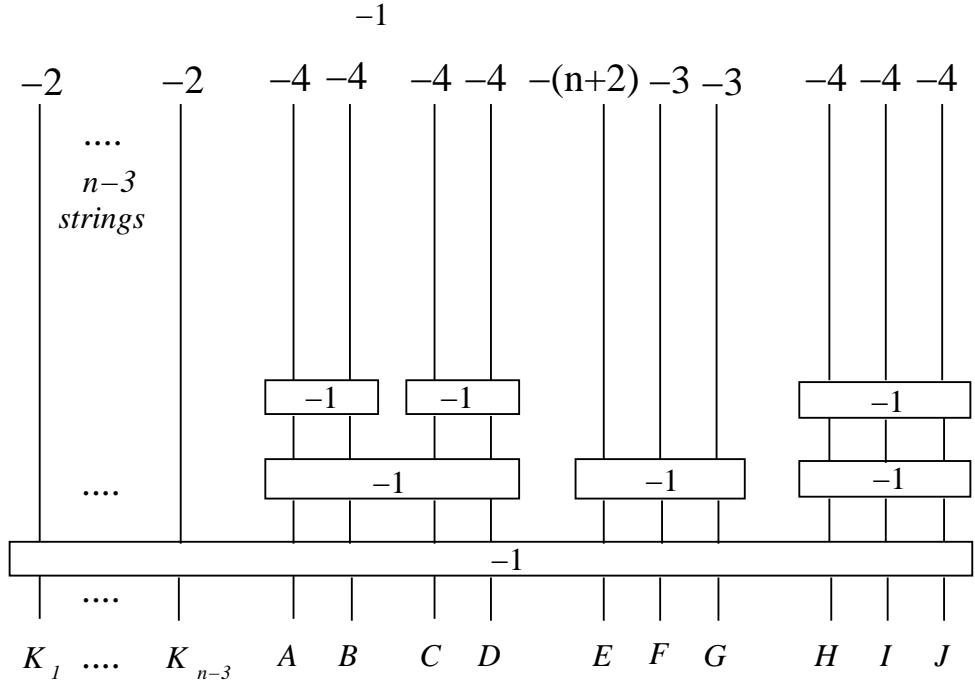}
\caption{Boxes in the diagram mean full negative twists.}
\label{f:d}
\end{figure}

\section{The nonexistence of rational homology disk smoothings}
\label{sec:four}
Next we will demonstrate how the explicit topological description of
the dual $D$ can be applied to study smoothings of surface
singularities.  We start with a simple observation providing an
obstruction for a 3--manifold to bound a rational homology disk,
i.e. a 4--manifold $V$ with $H_* (V; \bfq ) =H_* (D^4;\bfq )$.
\begin{thm}\label{t:forbid}
Suppose that the rational homology 3--sphere $-Y$ is the boundary of a
compact 4--manifold $D$ with the property that $\rank H_2(D; \Z )=n$
and that the intersection form $(H_2(D; \Z ), Q_D)$ does not embed
into the negative definite diagonal lattice $n\langle -1 \rangle $ of
the same rank.  Then $Y$ cannot bound a rational homology disk.
\end{thm}
\begin{proof}
  Suppose that such a rational homology disk $V$ exists; then $Z=V\cup _YD$ is a
  closed, negative definite 4--manifold. By Donaldson's Theorem \cite{Don} the
  intersection form of $Z$ is diagonalizable over $\Z$, and by our assumption
  on $V$ we get that $\rank H_2(Z; \Z )=\rank H_2 (D; \Z )=n $. Since $H_2 (D;
  \Z )\subset H_2(Z; \Z )$ does not embed into $n\langle -1\rangle$, we get a
  contradiction, implying the result. 
\end{proof}

Consider now the plumbing graph $\Gamma _n$ of Figure~\ref{f:int},
and denote the corresponding 3--manifold by $Y_n$.
\begin{prop}\label{p:nobou} 
The 3--manifold $Y_n$ does not bound a rational homology disk
4--manifold once $n\geq 7$.
\end{prop}

\begin{rem}
{\rm Notice that elements of this family pass all the tests provided
  by \cite{SSW} since these graphs are elements of the family
  ${\mathcal {A}}$ of \cite{SSW}: change the framing of the single
  $(-4)$--framed vertex with valency two to $(-1)$ and blow down the
  graph until it becomes the defining graph of the family ${\mathcal
    {A}}$. Also, using the algorithm described, e.g. in \cite{S} it is
  easy to see that $\det \Gamma _n \equiv n$ (mod 2), hence for odd
  $n$ the 3--manifold $Y_n$ admits a unique spin structure. The
  corresponding Wu class can be given by the $(-3)$--framed vertex of
  valency three, the unique $(-4)$--framed vertex on the long chain
  and then every second $(-2)$--framed vertex. A simple count then
  shows that for $n$ odd we have that ${\overline {\mu }}(Y_n)=0$,
  hence the result of \cite{S} provides no obstruction for a rational
  homology disk smoothing. (For the terminology used in the above
  argument, see \cite{S}.)}
\end{rem}
 
\begin{prop}\label{p:notembed}
The lattice determined by the intersection form of the dual $D_n$
given by Figure~\ref{f:d}, for $n \geq 7$, does not embed into the
same rank negative definite diagonal lattice.
\end{prop}
\begin{proof}
The labels on the components of the braid in Figure~\ref{f:d} will be
used to represent the corresponding basis elements for the lattice
determined by the intersection form of $D_n$. The rank is $n+7$. Let
$E = \{e_1, \ldots, e_{n+7}\}$ be the standard basis for the negative
definite diagonal lattice of rank $n+7$, so $e_i \cdot e_j =
-\delta_{ij}$. Suppose that the lattice for $D_n$ does embed into the
definite diagonal lattice. Then without loss of generality, since $K_i
\cdot K_i = -2$ and $K_i \cdot K_j =-1$ otherwise, we may assume that
$K_i = e_1 + e_{10+i}$. Furthermore, without loss of generality we may
assume that every other one of the basis elements $A, B, \ldots, J$ is
of the form $e_1 + x$ where $x$ is an expression in $e_2, \ldots,
e_{10}$. Thus each basis element whose square is $-3$ (i.e. $F$ and
$G$) must be of the form $e_1 \pm u \pm v$ where $u$ and $v$ are
distinct elements of the set $\{ e_2, \ldots, e_{10}\}$. Each element
whose square is $-4$ (i.e. $A$, $B$, $C$, $D$, $H$, $I$ and $J$) must
be of the form $e_1 \pm q \pm r \pm s$ where $q$, $r$ and $s$ are
distinct elements of the set $\{e_2, \ldots, e_{10}\}$.

Now we can assume that
$F = e_1 + e_2 + e_3$ and $G = e_1 + e_2 + e_4$
(noting that $F \cdot G = -2$). Then we note that none of the
expressions for $A,B,C,D,H,I$ or $J$ can contain $e_2, e_3$ or $e_4$
for the following reason: For each of $X = A,B,C,D,H,I,J$ there is
another basis element $Y$ from this set such that $X \cdot Y = -3$
while $X \cdot X = Y \cdot Y = -4$. Thus if we write $X = e_1 + \alpha
a + \beta b + \gamma c$ with $a, b, c \in E$ and $\alpha,\beta,\gamma
\in \{-1,1\}$, then $Y$ must be $Y = e_1 + \alpha a + \beta b + \delta
d$, with $d \in E$ and $\delta \in \{-1,1\}$, where $a$, $b$, $c$ and
$d$ are distinct elements from the set $\{ e_2, \ldots, e_{10}\}$. Now
noting that $X \cdot F = X \cdot G = Y \cdot F = Y \cdot G = -1$, we
see that if $a = e_2$ then $b,c,d$ must be in $\{e_3, e_4\}$ which
cannot happen because $b$, $c$ and $d$ must be distinct. Similarly $b$
cannot be $e_2$. If $a = e_3$ then $b$ or $c$ must be $e_2$, but we
have just seen that it cannot be $b$, so $c = e_2$. But the same
argument also shows that $d = e_2$, but $c \neq d$. Similarly we can
rule out $a = e_4$ and also $b = e_3$ and $b = e_4$. But if one of $c$
or $d$ is in the set $\{e_2, e_3, e_4\}$ then one of $a$ or $b$ must
also be, so finally we see that none of them can be.

Thus we can now take $H = e_1 + e_5 + e_6 + e_7$. There are then two
possibilities for $I$ and $J$ (up to relabelling the members of the
sets $\{e_8,e_9,e_{10}\}$ and $\{e_5,e_6,e_7\}$).

{\bf Case I:} $I = e_1+e_5+e_6+e_8$ and $J=e_1+e_5+e_6+e_9$.  In this
case we can see that $A,B,C$ and $D$ cannot contain $e_7$, $e_8$ or
$e_9$. So then the only remaining possibilities are all equivalent
(after changing signs of basis elements in $E$) to
$A=e_1+e_5-e_6+e_{10}$, but then we can not find any candidates for
$B$ which give $A \cdot B = -3$. This rules out Case I.

{\bf Case II:} $I = e_1+e_5+e_7+e_8$ and $J=e_1+e_5+e_6+e_8$. 
To rule out this case, write $A = e_1 + \alpha a + \beta b + \gamma c$,
$a,b,c \in \{ e_5,e_6,e_7,e_8,e_9,e_{10}\}$ and $\alpha, \beta, \gamma
\in \{-1.1\}$. In order to have $A \cdot H = -1$, either $0$ or $2$ of
$a,b,c$ must be in the set $\{e_5,e_6,e_7\}$, but not $1$ or $3$ of
them. Similarly, using $A \cdot I = -1$, either $0$ or $2$ must be in
$\{e_5,e_7,e_8\}$, and using $A \cdot J = -1$, either $0$ or $2$ must
be in $\{e_5,e_6,e_8\}$. If it is $0$ in one of these cases it must be
$0$ for all three, but that leaves only $e_9$ and $e_{10}$ for $a$,
$b$ and $c$, an impossibility. Thus it is $2$ in each case. We cannot
have one of them to be $e_5$, because then we could not have exactly $2$
from all three sets. So we must have $a=e_6$, $b=e_7$, $c=e_8$. But
exactly the same argument holds for $B$, and we can never get $A \cdot
B = -3$. Thus Case II is ruled out, concluding the proof of the
proposition. 
\end{proof}

\begin{proof}[Proof of Proposition~\ref{p:nobou}]
Combine Theorem~\ref{t:forbid} and Proposition~\ref{p:notembed}.
\end{proof}

\begin{cor}
  Suppose that $(S_{\Gamma}, 0)$ is an isolated surface singularity with
  resolution graph given by Figure~\ref{f:int}. If $n \geq 7$, then
  $(S_{\Gamma}, 0)$ admits no rational homology disk smoothing, i.e., it has
  no smoothing $V$ with $H_*(V; \bfz )=H_* (D^4; \bfz )$. 
\end{cor}

\end{document}